\documentclass[a4paper,10pt]{amsart}
\usepackage{amsmath, amssymb, amsthm}
\usepackage{verbatim}
\usepackage{hyperref}
\usepackage[all]{xy}

\newcounter{theorem}
\newtheorem{thm}[theorem]{\sc Theorem}
\newtheorem{lemma}[theorem]{\sc Lemma}
\newtheorem{prop}[theorem]{\sc Proposition}

\newtheorem{defn}[theorem]{\sc Definition}

\newtheorem{conjecture}[theorem]{\sc Conjecture}

\theoremstyle{remark}
\newtheorem*{remark*}{\sc Remark}
\newtheorem{remark}[theorem]{\sc Remark}

\numberwithin{equation}{section}
\numberwithin{theorem}{section}

\newcommand{\e}{\epsilon}
\newcommand{\dl}{\delta}

\newcommand{\C}{\mathbb{C}}

\renewcommand{\setminus}{\backslash}
\newcommand{\K}{\mathcal{K}}
\newcommand{\tens}{\otimes}
\newcommand{\dsum}{\oplus}
\newcommand{\bigdsum}{\bigoplus}
\newcommand{\dunion}{\amalg}

\renewcommand{\emptyset}{\varnothing}
\newcommand{\iso}{\cong}
\newcommand{\F}{\mathcal{F}}

\newcommand{\id}{\mathrm{id}}

\newcommand{\jsZ}{\mathcal{Z}}
\newcommand{\cO}{\mathcal{O}}
\newcommand{\dn}{\mathrm{dim}_{\mathrm{nuc}}}
\newcommand{\dr}{\mathrm{dr}\,}
\newcommand{\ev}{\mathrm{ev}}

\newcommand{\labelledthing}[2]{\hspace{4pt}\buildrel {#2} \over #1 \hspace{3pt}} 

\newcommand{\labelledrightarrow}{\labelledthing{\longrightarrow}}

\newcommand{\ccite}[2]{\cite[#1]{#2}}

\begin{document}

\title{\sc Decomposition rank of $\jsZ$-stable $\mathrm{C}^*$-algebras}
\author{Aaron Tikuisis}
\address{\hskip-\parindent
	Mathematisches Institut der WWU M\"unster, Einsteinstra{\ss}e 62, 48149 M\"unster, Germany.}
	\email{a.tikuisis@uni-muenster.de}
	\author {Wilhelm Winter}
	\email{wwinter@uni-muenster.de}

	\date{\today}

	\thanks{Both authors were supported by DFG (SFB 878).  W.W.\ was also supported by EPSRC (grants No.\ EP/G014019/1 and No.\ EP/I019227/1).}


	\begin{abstract}
	We show that $\mathrm{C}^{*}$-algebras of the form $C(X) \otimes \mathcal{Z}$, where $X$ is compact and Hausdorff and $\mathcal{Z}$ denotes the Jiang--Su algebra, have decomposition rank at most $2$. 
	This amounts to a dimension reduction result for $\mathrm{C^{*}}$-bundles with sufficiently regular fibres. It establishes an important case of  a conjecture on the fine structure of nuclear $\mathrm{C}^{*}$-algebras of Toms and the second named author, even in a nonsimple setting, and gives evidence that the topological dimension of noncommutative spaces is governed by fibres rather than base spaces. 
	\end{abstract}

	\maketitle

	\section{Introduction}

	\noindent
	The structure and classification theory of nuclear $\mathrm{C}^{*}$-algebras has seen rapid pro\-gress in recent years, largely spurred by the subtle interplay between certain topological and algebraic regularity properties, such as finite topological dimension, tensorial absorption of suitable strongly self-absorbing $\mathrm{C}^{*}$-algebras and order completeness of homological invariants, see \cite{ElliottToms} for an overview. In the simple and unital case, these relations were formalized by A.~Toms and the second named author as follows: 

	\begin{conjecture}
	\label{CTW}
	For a separable, simple, unital, nonelementary, stably finite and nuclear $\mathrm{C}^{*}$-algebra $A$, the following are equivalent:

	\begin{enumerate}
	\item $A$ has finite decomposition rank, $\mathrm{dr}\, A < \infty$,
	\item $A$ is $\mathcal{Z}$-stable, $A \cong A \otimes \mathcal{Z}$,
	\item $A$ has strict comparison of positive elements.
	\end{enumerate}
	\end{conjecture}

	Here, decomposition rank is a notion of noncommutative topological dimension introduced in \cite{KirchbergWinter:CovDim}, $\mathcal{Z}$ denotes the Jiang--Su algebra introduced in \cite{JiangSu} and strict comparison essentially means that positive elements may be compared  in terms of tracial values of their support projections, cf.\ \cite{Rordam:ICM}. If one drops the finiteness assumption on $A$, one should replace (i) by 
	\begin{enumerate}
	\item[(i')] $A$ has finite nuclear dimension, $\dim_{\mathrm{nuc}}A <\infty$,
	\end{enumerate}
	where nuclear dimension \cite{WinterZacharias:NucDim} is a variation of the decomposition rank which can have finite values also for infinite $\mathrm{C}^{*}$-algebras. 

	The conjecture still makes sense in the nonsimple situation, provided one asks $A$ to have no elementary subquotients (this is a minimal requirement for $\mathcal{Z}$-stability); one also has to be slightly more careful about the definition of comparison in this case. 

	Nuclearity in this context manifests itself most prominently via approximation properties with particularly nice completely positive maps \cite{CSSWW:perturbations,HKW}.

	Conjecture~\ref{CTW} has a number of important consequences for the structure of nuclear $\mathrm{C}^{*}$-algebras and it has turned out to be pivotal for many recent classification results, especially in view of the examples given in \cite{Ror:simple, Toms:annals, Vil:sr}. Moreover, it highlights the striking analogy between the classification program for nuclear $\mathrm{C}^{*}$-algebras, cf.\ \cite{Ell:classprob}, and Connes' celebrated classification of injective $\mathrm{II}_{1}$ factors  \cite{Con:class}.

	Implications (i), (i') $\Longrightarrow$ (ii) $\Longrightarrow$ (iii) of Conjecture \ref{CTW} are by now known to hold in full generality \cite{Rordam:Z,Winter:drZstable,Winter:pure}; (iii) $\Longrightarrow$ (ii) has been established under certain additional structural hypotheses \cite{MatuiSato:Comp,Winter:pure}, all of which in particular guarantee sufficient divisibility properties.

	Arguably, it is (ii) $\Longrightarrow$ (i) which remains the least well understood of these implications. While there are promising partial results available \cite{Lin:LocalAH, Winter:pure, WinterZacharias:NucDim}, all of these factorize through classification theorems of some sort. This in turn makes it hard to explicitly identify the origin of finite dimensionality. 

	In the simple purely infinite (hence $\mathcal{O}_{\infty}$-stable, hence $\mathcal{Z}$-stable \cite{Kir:CentralSequences,KirRor:pi2}) case, one has to use Kirchberg--Phillips classification \cite{Kir:ICM, KirchbergPhillips:ExactEmbedding}  as well as a range result providing models to exhaust the invariant \cite{Ror:encyc} and then again Kirchberg--Phillips classification to show that these models have finite nuclear dimension \cite{WinterZacharias:NucDim}. 

	In the simple stably finite case, at this point only approximately homogeneous (AH) algebras or approximately subhomogeneous (ASH) algebras for which projections separate traces are covered \cite{Lin:AsympClassification, LinNiu:KKlifting, Winter:drSH, Winter:localizing}. (This approach also includes crossed products associated to uniquely ergodic minimal dynamical systems \cite{TomsWinter:rigidity, TomsWinter:MinClassification}.) While both of these classes after stabilizing with $\mathcal{Z}$ can by now be shown directly to consist of TAI and TAF algebras \cite{Lin:LocalAH}, again finite topological dimension will only follow from classification results \cite{ElliottGongLi:AHclassification, Lin:AsympClassification, Toms:rigidity, Winter:drZstable} and after comparing to models which exhaust the invariant \cite{Elliott:ashrange,Villadsen:AHrange}; see also \cite{Ror:encyc} for an overview.  

	Once again, the classification procedure does not make it entirely transparent where the finite topological dimension comes from, but at least Elliott--Gong--Li classification of simple AH algebras (of very slow dimension growth -- later shown to be equivalent to slow dimension growth and to $\mathcal{Z}$-stability \cite{Winter:pure}) heavily relies on Gong's deep dimension-reduction theorem  \cite{Gong:SimpleReduction}. Gong gives an essentially explicit way of replacing a given AH limit decomposition with one of low topological dimension. However, this method is technically very involved and requires both simplicity and the given inductive limit decomposition. It does not fully explain to what extent the two are necessary; in particular, it is in principle conceivable that a decomposition similar to that of Gong exists for algebras of the form $C(X) \otimes \mathcal{Q}$ (with $\mathcal{Q}$ being the universal UHF algebra). 

	In this article we show how finite topological dimension indeed arises for algebras of this type; in fact, we are able to cover algebras of the form $C(X) \otimes \mathcal{Z}$, and hence also locally homogeneous $\mathcal{Z}$-stable $\mathrm{C}^{*}$-algebras (not necessarily simple, or with a prescribed inductive limit structure). We hope our argument will shed new light on the conceptual reasons why finite topological dimension should arise in the presence of sufficient $\mathrm{C}^{*}$-algebraic regularity. Our method is based on approximately embedding the cone over the Cuntz algebra $\mathcal{O}_{2}$ into tracially small subalgebras of the algebra in question; these play a similar role as the small corners used in the definition of TAF algebras \cite{Lin:TAFduke} or the small hereditary subalgebras in property SI  \cite{MatuiSato:Comp}. We mention that we only obtain (a strong version of) finite decomposition rank, whereas Gong's reduction theorem yields an inductive limit decomposition; however, for many purposes finite decomposition rank is sufficient, cf.\ \cite{TomsWinter:rigidity, Winter:drZstable}.

	In \cite{KirchbergRordam:pi3}, algebras of the form $C(X) \otimes \mathcal{O}_{2}$ were shown to be approximated by algebras of the form $C(\Gamma) \otimes \mathcal{O}_{2}$ with $\Gamma$ one-dimensional. Since $\mathcal{O}_{2}$ is by now known to have finite nuclear dimension \cite{WinterZacharias:NucDim}, this may be regarded as strong evidence that the topological dimension of a $\mathrm{C}^{*}$-bundle depends on the noncommutative size of the fibres more than the size of the base space. (A somewhat similar phenomenon was already observed for stable rank by Rieffel \cite{Rfl:sr}.)

It is remarkable that \cite{KirchbergRordam:pi3} does not rely on a classification result in any way. It does, however, mix commutativity (of the structure algebra) and pure infiniteness (of the fibres). 

It is not clear from \cite{KirchbergRordam:pi3} whether such a dimension type reduction also occurs in the setting of stably finite fibres. In the present article we show that it does, by developing a method to transport \cite{KirchbergRordam:pi3} to the situation where the fibres are UHF algebras (to pass to the case where each fibre is $\mathcal{Z}$ then requires a certain amount of additional machinery -- at least if one wants to increase the dimension by no more than one). The crucial concept to link purely infinite and stably finite fibres is quasidiagonality of the cone over $\mathcal{O}_{2}$, discovered by Voiculescu and by Kirchberg \cite{Kirchberg:nonsemisplit,Voiculescu:Quasidiagonal}.  

One should mention that the fact that the fibres are specific strongly self-absorbing algebras in both \cite{KirchbergRordam:pi3} and in our result plays an important, but in some sense secondary role: In \cite{KirchbergRordam:pi3} (combined with \cite{WinterZacharias:NucDim}) one can replace $\mathcal{O}_{2}$ with $\mathcal{O}_{\infty}$, or in fact with any UCT Kirchberg algebra, and still arrive at finite nuclear dimension. More generally, our result yields the respective statement if the fibres have finite nuclear dimension and are $\mathcal{Z}$-stable, e.g.\ in the simple, nuclear, classifiable case.

While at the current stage we only cover the case of highly homogeneous bundles (in fact, $C(X) \otimes \mathcal{Z}$ -- but it is an easy exercise to pass to more general bundles with Hausdorff spectrum from here), it will be an important task to handle bundles with non-Hausdorff spectrum, e.g.\ $B \otimes \mathcal{Z}$ with $B$ subhomogeneous, in order to also cover transformation group $\mathrm{C}^{*}$-algebras. This will be pursued in subsequent work by combining our technical Lemma \ref{DimDropUnit} with the methods of \cite{Winter:drSH}; in preparation, we have stated \ref{DimDropUnit} in a form slightly more general than necessary for the current main result, Theorem~\ref{AHZdr}.

\section{Decomposition rank of homomorphisms}
\label{DRSect}

\noindent
In this section, we introduce the notions of decomposition rank and nuclear dimension of $^*$-homomorphisms, building naturally on the respective notions for $\mathrm{C}^*$-algebras, just as nuclearity for $^*$-homomorphisms arises from the completely positive approximation property for $\mathrm{C}^*$-algebras. We first recall the notion of completely positive contractive (c.p.c.) order zero maps, cf.\ \cite{Winter:CovDim1}.

\begin{defn}
Let $A,B$ be $\mathrm{C}^*$-algebras and let $\phi:A \to B$ be a c.p.c.\ map.
We say that $\phi$ has order zero if it preserves orthogonality in the sense that if $a,b \in A_+$ satisfy $ab=0$ then $\phi(a)\phi(b)=0$.
\end{defn}

\begin{defn}\label{DimNucForHomos}
Let $\alpha:A \to B$ be a $^*$-homomorphism of $\mathrm{C}^*$-algebras.
We say that $\alpha$ has decomposition rank at most $n$, and denote $\dr(\alpha) \leq n$, if for any finite subset $\F \subset A$ and any $\e > 0$, there exists a finite dimensional $\mathrm{C}^*$-algebra $F$ and c.p.c.\ maps
\[ \psi: A \to F \text { and } \phi: F \to B \]
such that $\phi$ is $(n+1)$-colourable, in the sense that we can write
\[ F=F^{(0)} \dsum \ldots \dsum F^{(n)} \]
and $\phi|_{F^{(i)}}$ has order zero for all $i$, and such that $\phi\psi$ is point-norm close to $\alpha$, in the sense that for $a \in \F$,
\[ \|\alpha(a) - \phi\psi(a)\| < \e. \]

We may define nuclear dimension of $\alpha$ similarly (and denote $\dn(\alpha) \leq n$), where instead of requiring that $\phi$ is contractive, we only ask that $\phi|_{F^{(i)}}$ is contractive for each $i$.
\end{defn}

\begin{remark}
Notice that the decomposition rank (respectively nuclear dimension) of a $\mathrm{C}^*$-algebra,  as defined in \ccite{Definition 3.1}{KirchbergWinter:CovDim} (respectively \ccite{Definition 2.1}{WinterZacharias:NucDim}) is just the decomposition rank (respectively nuclear dimension) of the identity map.
\end{remark}

The following generalizes some permanence properties for decomposition rank and nuclear dimension of $\mathrm{C}^*$-algebras.
Proofs are omitted, as they are essentially the same as those found in \cite{KirchbergWinter:CovDim, Winter:CovDim1, WinterZacharias:NucDim}.

\begin{prop}
\label{drPermanence}
Let $A,B$ be $\mathrm{C}^*$-algebras and let $\alpha:A \to B$ be a $^*$-homomorphism.
\begin{enumerate}
\item
Suppose that $A$ is locally approximated by a family of $\mathrm{C}^{*}$-subalgebras $(A_\lambda)_{\Lambda}$, in the sense that for every finite subset $\F \subset A$ and every tolerance $\e > 0$, there exists $\lambda$ such that $\F \subset_\e A_\lambda$.
Then
\[ \dr(\alpha) \leq \sup_{\Lambda} \dr(\alpha|_{A_\lambda}) \]
and
\[ \dn(\alpha) \leq \sup_{\Lambda} \dn(\alpha|_{A_\lambda}). \]
\item
If $C \subset A$ is a hereditary $\mathrm{C}^{*}$-subalgebra, then  
\[ \dr(\alpha_{C}) \leq \dr(\alpha)\]
and
\[ \dn(\alpha_{C}) \leq \dn(\alpha), \]
where $\alpha_{C}:= \alpha|_{C}: C \to \mathrm{her}(\alpha({C}))$.
\end{enumerate}
\end{prop}

When computing the decomposition rank (or nuclear dimension), it is often convenient to replace the codomain by its sequence algebra, defined to be
\[ \textstyle{A_\infty := \big(\prod_{\mathbb{N}} A\big) / \big(\bigdsum_{\mathbb{N}} A\big).} \]
We shall  denote by 
\[
\textstyle
\pi_\infty: \prod_{\mathbb{N}} A \to A_\infty
\]
the quotient map, and by $\iota_\infty:A \to A_\infty$ the canonical embedding as constant sequences.

\begin{prop}
\label{AsympSequence}
Let $\alpha:A \to B$ be a $^*$-homomorphism.

Then,
\[ \dr(\alpha) = \dr(\iota_\infty \circ \alpha) \]
and
\[\dn(\alpha) = \dn(\iota_\infty \circ \alpha). \]
\end{prop}

\begin{proof}
Straightforward, using stability of the relations defining c.p.c.\ order zero maps on finite dimensional domains \cite{KirchbergWinter:CovDim}.
\end{proof}

\begin{prop}\label{SSA-TensorFormula}
Let $\mathcal{D}$ be a strongly self-absorbing $\mathrm{C}^*$-algebra, as defined in \cite{TomsWinter:ssa}, and let $A$ be a $\mathcal{D}$-stable $\mathrm{C}^*$-algebra. 

Then
\[ \dr(A) = \dr(\id_{A} \otimes 1_{\mathcal{D}})  \]
and
\[\dn(A) = \dn(\id_{A} \otimes 1_{\mathcal{D}}). \]
\end{prop}

\begin{proof}
This follows easily from the fact that $\id_{\mathcal{D}}$ has approximate factorizations of the form
\[ \mathcal{D} \labelledrightarrow{\id_{\mathcal{D}} \tens 1_{\mathcal{D}}} \mathcal{D} \tens \mathcal{D} \labelledrightarrow{\phi} A \tens \mathcal{D}, \]
where $\phi$ is a $^*$-isomorphism.
\end{proof}

\section{\texorpdfstring{$C(X)$}{C(X)}-algebras and decomposition rank}
\label{CXAlgSect}

\noindent
For a locally compact Hausdorff space $X$, a $C_{0}(X)$-algebra is a $\mathrm{C}^*$-algebra $A$ equipped with a nondegenerate $^*$-homomorphism $C_{0}(X) \to Z\mathcal{M}(A)$, called the structure map \ccite{Definition 1.5}{Kasparov:KK}. 
Here, $\mathcal{M}(A)$ refers to the multiplier algebra of $A$ and $Z\mathcal{M}(A)$ to its centre; note that if $A$ is unital, then so is the structure map. In this section, we study the decomposition rank of such structure maps. 
Proposition \ref{AbelianInclusionDim} below is reminiscent of \ccite{Proposition 2.19}{Winter:CovDim1} which shows that the completely positive rank of $C(X)$ equals the covering dimension of $X$.

\begin{defn}
Let $A$ be a $C(X)$-algebra and let $a \in A$.
Define the support of $a$ to be the smallest closed set $F \subset X$ such that $ag = 0$ whenever $g \in C_0(X \setminus F) \subset C(X)$. (This is easily seen to be well-defined.)
\end{defn}

\begin{prop}
\label{AbelianInclusionDim}
Let $X$ be a compact Hausdorff space, and let $A$ be a unital $C(X)$-algebra with structure map $\iota:C(X) \to Z(A)$.

The following are equivalent.
\begin{enumerate}
\item $\dr(\iota) \leq n$.
\item $\dn(\iota) \leq n$.
\item The definition of $\dr(\iota) \leq n$ holds with the additional requirements that $F$ is abelian and $\psi$ is a unital $^*$-homomorphism.
\item For any finite open cover $\mathcal{U}$ of $X$, any $\e > 0$, and any $b \in C(X)$, there exists an $(n+1)$-colourable $\epsilon$-approximate finite partition of $b$; that is, positive elements $b_j^{(i)} \in A$ for $i=0,\dots,n, j=1,\dots,r$, such that
\begin{enumerate}
\item for each $i$, the elements $b_1^{(i)},\dots,b_r^{(i)}$ are pairwise orthogonal,
\item for each $i,j$, the support of $b_j^{(i)}$ is contained in some open set in the given cover $\mathcal{U}$,
\item $\| \sum_{i,j} b_j^{(i)} - \iota(b) \| \leq \e$.
\end{enumerate}
\end{enumerate}
\end{prop}

\begin{proof}
(iii) $\Rightarrow$ (i) $\Rightarrow$ (ii) is obvious.

(ii) $\Rightarrow$ (iv):
Let us first assume $b=1$. Let $\F$ be a finite partition of unity such that, for each $f \in \F$, there exists $U_f \in \mathcal{U}$ such that $\mathrm{supp}\, f \subset U_f$.
Set 
\begin{equation}
\label{AbelianInclusionDim-etaDefn}
\eta := \frac{\e}{2|\F|(n+1)}.
\end{equation}
Use $\dn(\iota) \leq n$ to obtain 
\[ C(X) \labelledrightarrow{\psi} F^{(0)} \dsum \ldots \dsum F^{(n)} \labelledrightarrow{\phi} A \]
such that $\psi$ is c.p.c., $\phi|_{F^{(i)}}$ is c.p.c.\ and order zero for all $i=0,\dots,n$, $\phi(\psi(f)) =_\eta f$ for $f \in \F$, and $\phi(\psi(1)) =_{\e/2} 1$.
Let 
\[
F^{(i)} = \bigdsum_{j=1}^{r_i} M_{m(i,j)}.
\] 
(By throwing in some zero summands if necessary, we may as well assume all the $r_{i}$'s to be equal.) 

For each $i=0,\dots,n$ and $j=1,\dots,r_i$, we set
\[ a_j^{(i)} := \left(\phi(\psi(1_{C(X)})1_{M_{m(i,j)}}) - \frac{\e}{2(n+1)}\right)_+. \]
For each $i$, since $\phi|_{F^{(i)}}$ is order zero, $a_1^{(i)},\dots,a_{r_i}^{(i)}$ are orthogonal.
We estimate
\[
\begin{array}{rll}
1 &=_{\e/2}& \phi(\psi(1)) \\
&=& \displaystyle \sum_{i=0}^n \sum_{j=1}^{r_i} \phi(\psi(1)1_{M_{m(i,j)}}) \\
& =_{\frac{(n+1)\e}{2(n+1)}}& \displaystyle\sum_{i,j} a_j^{(i)},
\end{array}
\]
where the last approximation is obtained using the fact that the inner summands are orthogonal.

Lastly, we must verify that each $a_j^{(i)}$ has support contained in an open set from the cover $\mathcal{U}$.
Fix $i$ and $j$.
Let $f_{i,j} \in \F$ maximize $f \mapsto \|\psi(f)1_{M_{m(i,j)})}\|$.
We shall show that the support of $a_j^{(i)}$ is contained in the support of $f_{i,j}$ by showing that $a_j^{(i)}|_K = 0$, where
\[ K := \{x \in X: f_{i,j} = 0\}. \]

Since $1 = \sum_{f \in \F} f$, we must have 
\begin{equation}
 \|\psi(f_{i,j})1_{M_{m(i,j)}}\| \geq \frac1{|\F|}\|\psi(1)1_{M_{m(i,j)}}\|.
\label{AbelianInclusionDim-ProofNormIneq1}
\end{equation}
Noting that
\[
\begin{array}{rll}
\displaystyle f_{i,j} &=_{\eta}&  \phi(\psi(f_{i,j})) \\
\displaystyle&\geq& \phi(\psi(f_{i,j})1_{M_{m(i,j)}}),
\end{array}
\]
we must have
\begin{equation}
\label{AbelianInclusionDim-ProofNormIneq2}
\|\phi(\psi(f_{i,j})1_{M_{m(i,j)}})|_K\| \leq \eta.
\end{equation}
We get
\begin{eqnarray*}
\displaystyle \|\phi(\psi(1)1_{M_{m(i,j)}})|_K\| &=& \|\phi(1_{M_{m(i,j)}})|_K\|\,\|\psi(1)1_{M_{m(i,j)}}\| \\
\displaystyle &\stackrel{\eqref{AbelianInclusionDim-ProofNormIneq1}}{\leq}& \|\phi(1_{M_{m(i,j)}})|_K\|\,|\F|\,\|\psi(f_{i,j})1_{M_{m(i,j)}}\| \\
\displaystyle &= &\|\phi(\psi(f_{i,j})1_{M_{m(i,j)}})|_K\| \\
\displaystyle &\stackrel{\eqref{AbelianInclusionDim-etaDefn}, \eqref{AbelianInclusionDim-ProofNormIneq2}}{\leq}& \frac\e{2(n+1)},
\end{eqnarray*}
where the equalities on the first and third line are obtained as noted in the proof of \ccite{Proposition 5.1}{KirchbergWinter:CovDim} (6$^\text{th}$ line from the bottom of page 79); therefore, $a_j^{(i)}|_K = 0$, as required.

If $b$ is not the unit, we may still assume that $\|b\| \leq 1$ and use the argument above to obtain an $(n+1)$-colourable approximate partition of unity $(a_j^{(i)})$ subordinate to $\mathcal{U}$.
Then simply set $b_j^{(i)} = ba_j^{(i)}$.

(iv) $\Rightarrow$ (iii):
It will suffice to prove the condition in (iii) assuming that $\F$ consists of self-adjoint contractions.

Take an open cover $\mathcal{U}$ of $X$ along with points $x_{U} \in U$ for every $U \in \mathcal{U}$  such that, for any $f \in \F$, $U \in \mathcal{U}$ and $x \in U$,
\begin{equation}
\label{AbelianInclusionDim-ApproxConst}
 |f(x) - f(x_U)| < \frac\e2.
\end{equation}
Use (iv) with $b=1$ to find an $(n+1)$-colourable $\frac\e2$-approximate partition of unity 
\[
(a_j^{(i)})_{i=0,\ldots,n;\, j=1,\ldots,r}
\]
subordinate to $\mathcal{U}$.
By a standard rescaling argument, we may assume that $\sum a_j^{(i)} \leq 1$.
For each $i,j$, let $U(i,j) \in \mathcal{U}$ be such that $\mathrm{supp}\, a_j^{(i)} \subset U(i,j)$.

Define $\psi:C(X) \to (\mathbb{C}^r)^n$ by
\[ \psi(f) = (f(x_{U(i,j)}))_{i=0,\dots,n;\, j=1,\dots,r} \]
and define $\phi:(\mathbb{C}^r)^n \to C(X,A)$ by
\[ \phi(\lambda_{i,j})_{i=0,\dots,n;\, j=1,\dots,r} = \sum_{i,j} \lambda_{i,j} \cdot a_j^{(i)}. \]
Clearly, $\psi$ is a $^*$-homomorphism, while $\phi$ is c.p.c.\ and its restriction to each copy of $\mathbb{C}^{r}$ is order zero.

To verify that $\phi \circ \psi$ approximates $\theta$ in the appropriate sense, fix $f \in \F$ and $x \in X$.
We shall show that $\|\phi\psi(f)(x) - f(x)\| < \e$ (in the fibre $A(x)$).
Let
\[ S = \{(i,j) \in \{0,\dots,n\} \times \{1,\dots,r\}: x \in U(i,j)\}, \]
so that
\[ \phi(\psi(f))(x) = \sum_{(i,j) \in S} f(x_{U(i,j)}) \cdot a_j^{(i)}(x), \]
and
\[ 1 =_{\e/2} \sum_{(i,j) \in S} a_j^{(i)}(x). \]
By \eqref{AbelianInclusionDim-ApproxConst},
\begin{eqnarray*}
 (f(x)-\e/2) \cdot \sum_{(i,j) \in S} a_j^{(i)}(x) &\leq& \sum_{(i,j) \in S} f(x_{U(i,j)}) \cdot a_j^{(i)}(x) \\
&\leq& (f(x) + \e/2) \cdot \sum_{(i,j) \in S} a_j^{(i)}(x).
\end{eqnarray*}
It follows that
\[
\begin{array}{rll}
 \phi(\psi(f)) &=& \displaystyle \sum_{(i,j) \in S} f(x_{U(i,j)}) \cdot a_j^{(i)} \\
&=_{\e/2}& \displaystyle f(x) \cdot \sum_{(i,j) \in S} a_j^{(i)} \\
&=_{\e/2}& f(x),
\end{array}
\]
as required.
\end{proof}

\begin{prop}
\label{O2CtsField}
Let $X$ be a locally compact metrizable space with finite covering dimension, and let $A$ be a $C_0(X)$-algebra all of whose fibres are isomorphic to $\cO_2$.
Let $U \subset X$ be an open subset such that $\overline{U}$ is compact.

Then $C_0(U)A \iso C_0(U,\cO_2)$ as $C_0(U)$-algebras.
\end{prop}

\begin{proof}
\ccite{Theorem 1.1}{Dadarlat:CtsFields} says that $A|_{\overline{U}} \iso C(\overline{U},\cO_2)$, as $C(\overline{U})$-algebras.
Viewing $C_0(U)A$ as an ideal of $A|_{\overline{U}}$, the result follows.
\end{proof}

\section{Decomposition rank of \texorpdfstring{$C_0(X,\jsZ)$}{C(X,Z)}}
\label{MainThmSect}

\noindent
In this section, we prove our main result:

\begin{thm}
\label{AHZdr}
Let $A$ be a $\mathrm{C}^*$-algebra which is locally approximated by hereditary subalgebras of $\mathrm{C}^*$-algebras of the form $C(X,\mathcal{K})$, with $X$ compact Hausdorff. 

Then
\[ \dr(A \tens \jsZ) \leq 2. \]
In particular, any $\mathcal{Z}$-stable $\mathrm{AH}$ $\mathrm{C}^{*}$-algebra has decomposition rank at most $2$. 
\end{thm}

In our proof, we will make use of the huge amount of space provided by the noncommutative fibres in two ways. First, we exhaust the identity on $X$ by pairwise orthogonal functions up to a tracially small hereditary subalgebra. This will be designed to host an algebra of the form $C_{0}(Z) \otimes \mathcal{O}_{2}$, which is possible by quasidiagonality of the cone over $\mathcal{O}_{2}$. The first factor embedding of $C_{0}(Z)$ into the latter can be approximated by $2$-colourable maps as shown by Kirchberg and R{\o}rdam (see below). Together with the initial set of functions we obtain a $3$-colourable, hence $2$-dimensional, approximation of the first factor embedding of $C(X)$ into $C(X) \otimes \mathcal{Z}$. 

We will first carry out this construction with a UHF algebra in place of $\mathcal{Z}$; a slight modification will then allow us to pass to certain $C([0,1])$-algebras with UHF fibres, which immediately yields the general case.  

As noted above, a result of Kirchberg and R{\o}rdam \ccite{Proposition 3.7}{KirchbergRordam:pi3} on $1$-dimensional approximations in the case of $\mathcal{O}_{2}$-fibred bundles is a crucial ingredient; this in turn relies on the fact that the unitary group of $C(S^{1},\mathcal{O}_{2})$ is connected \cite{Cuntz:KOn}. We note the following direct consequence which is more adapted to our needs.

\begin{thm}\label{O2dimnuc}
For any locally compact Hausdorff space $X$, the decomposition rank of the first factor embedding $C_0(X) \to C_0(X,\cO_2)$ is at most one.
\end{thm}

\begin{proof}
Let us begin with the case that $X$ is compact and metrizable.
By \ccite{Proposition 3.7}{KirchbergRordam:pi3}, there exists a $^*$-subalgebra $A \subset C(X,\cO_2)$ which contains $C(X) \tens 1_{\cO_2}$ and is isomorphic to $C(Y)$ where $Y$ is compact metrizable with covering dimension at most one.
Therefore, the decomposition rank of the first factor embedding $C(X) \to C(X) \tens \cO_2$ is at most the decomposition rank of the inclusion $C(X) \tens 1_{\cO_2} \subset A$, which in turn is at most $\dr A \leq 1$.

For $X$ compact but not metrizable, $C(X)$ is locally approximated by finitely generated unital subalgebras, which are of the form $C(Y)$ where $Y$ is compact and metrizable.
Therefore by Proposition \ref{drPermanence} (i), the claim holds in this case too.

For the case that $X$ is not compact, we let $\tilde X$ denote the one-point compactification of $X$.
Then $C_0(X,\cO_2)$ is the hereditary subalgebra of $C(\tilde X, \cO_2)$ generated by $C_0(X)$, and therefore the result follows from Proposition \ref{drPermanence} (ii).
\end{proof}

\begin{remark}
The preceding result also implies that $\dn(A \tens \cO_2) \leq 3$ for $A$ as in Theorem \ref{AHZdr} -- this can be seen using Proposition \ref{SSA-TensorFormula}, \ccite{Theorem 7.4}{WinterZacharias:NucDim} and the analogue of \ccite{Proposition 2.3 (ii)}{WinterZacharias:NucDim}.
\end{remark}

In what follows, $D_n$ denotes the diagonal subalgebra of $M_n$.

\begin{lemma}
\label{TracialSplit}
Let $I_1,\dots,I_n \subset (0,1)$ be nonempty closed intervals and let $a_{1/2} \in C_0((0,1),D_n)_+$ be a function of norm $1$ such that for $t \in I_s$, the $s^{\text{th}}$ diagonal entry of $a_{1/2}(t)$ is $1$.

Then there exist $a_0,a_1,e_0,e_{1/2},e_1 \in C([0,1],D_n)_+$ such that
\begin{enumerate}
\item $e_0$ and $e_1$ are orthogonal,
\item $a_0+a_{1/2}+a_1 = e_0+e_{1/2}+e_1 = 1$,
\item for $i=0,1$, we have $a_i(i) = 1_n$,
\item $e_0,e_1$ act like a unit on $a_0,a_1$ respectively,
\item $a_{1/2}$ acts like a unit on $e_{1/2}$.
\end{enumerate}
\end{lemma}

\begin{proof}
Since $D_n \iso \C^{n}$, it suffices to work in one coordinate at a time -- that is to say, to assume that $n=1$.
Then define
\begin{align*}
a_0(x) &:= \begin{cases} 1-a_{1/2}(x), \quad &\text{if $x$ is to the left of $I_1$} \\
0, \quad &\text{otherwise;} \end{cases} \\
a_1(x) &:= \begin{cases} 1-a_{1/2}(x), \quad &\text{if $x$ is to the right of $I_1$} \\
0, \quad &\text{otherwise.} \end{cases}
\end{align*}
Note that since $a_{1/2} \equiv 1$ on $I_1$, these are continuous.
Now, we may find continuous orthogonal functions $e_0,e_1$ such that $e_0$ is $1$ to the left of $I_1$ and $e_1$ is $1$ to the right of $I_1$.
Finally, set $e_{1/2} := 1-(e_0+e_1)$.
Then (i), (ii), (iii) clearly hold by construction.
(iv) holds since each $a_i$ is nonzero only on one side of $I_1$, and the corresponding $e_i$ is identically $1$ on that side.
Likewise, (v) holds since $e_{1/2}$ is nonzero only on $I_1$, where $a_{1/2}$ is identically $1$.
\end{proof}

We mention the following well-known fact explicitly for convenience. Here, $\otimes$ denotes the minimal tensor product.

\begin{prop}\label{AsymptoticTensor}
Let $A_1,A_2,B_1,B_2$ be  $\mathrm{C}^*$-algebras, and suppose that $\phi^{(i)}:A_i \to (B_i)_\infty$ is a $^*$-homomorphism for $i=1,2$ with a c.p.\ lift $(\phi^{(i)}_k)_{\mathbb{N}}:A_i \to \prod_{\mathbb{N}} B_i$.

Then
\[ \phi_1 \tens \phi_2 = \pi_\infty \circ (\phi^{(1)}_k \tens \phi^{(2)}_k)_{\mathbb{N}}: A_1 \tens A_2 \to (B_1 \tens B_2)_\infty \]
is a $^*$-homomorphism.
\end{prop}

\begin{lemma}
\label{IntervalUnit}
Let $A$ be an infinite dimensional UHF algebra.

Then there exist positive orthogonal contractions 
\[ a_0,a_1 \in C([0,1],A)_\infty, \]
a $^*$-homomorphism
\[ \psi:C_0(Z, \cO_2) \to C_0((0,1),A)_\infty \]
for some locally compact, metrizable, finite dimensional space $Z$, and a positive element $c \in C_{\mathrm{c}}(Z, \C\cdot1_{\cO_2})$ such that $\psi(c)$ commutes with $a_0,a_1$,
\begin{equation}
\label{IntervalUnit-POU}
a_0 + a_1 + \psi(c) = 1,
\end{equation}
and for $i=0,1$, we have $a_i(i) = 1$.
In addition, there exist positive contractions $e_0,e_{1/2},e_1 \in C([0,1],A)_\infty$ such that
\begin{enumerate}
\item $e_0,e_1$ are orthogonal,
\item $e_0+e_{1/2}+e_1 = 1$,
\item $\psi(c)$ acts like a unit on $e_{1/2}$,
\item $e_i$ acts like a unit on $a_i$ for $i=0,1$,
\item $e_0,e_{1/2},e_1,a_0,a_1,\psi(c)$ all commute.
\end{enumerate}
\end{lemma}

\begin{proof}
Let $A = M_\mathfrak{n}$ where $\mathfrak{n}$ is the supernatural number
\[ n_1n_2\ldots. \]
Since the cone over $\mathcal{O}_{2}$ is quasidiagonal, cf.\  \cite{Voiculescu:Quasidiagonal} and  \ccite{Theorem 5.1}{Kirchberg:nonsemisplit}, there exists a sequence of c.p.c.\ maps
\[ \phi_k:C_0((0,1],\cO_2) \to M_{n_1\ldots n_k} \]
which are approximately multiplicative and approximately isometric, meaning that for all $a,b \in C_0((0,1],\cO_2)$,
\[ \|\phi_k(a)\phi_k(b) - \phi_k(ab)\| \to 0 \]
and
\[ \|\phi_k(a)\| \to \|a\| \]
as $k \to \infty$.
Fix a positive element 
\[
d \in C_{\mathrm{c}}((0,1], \C\cdot1_{\cO_2})
\]
of norm $1$.

For each $k$, let $\lambda_k$ denote the greatest eigenvalue of $\phi_k(d)$.
Note that 
\[
\lambda_k = \|\phi_k(d)\| \to 1
\]
as $k \to \infty$.

Fix $k$ for a moment and let $l = n_1\ldots n_k$.
Let 
\[
I_1,\dots,I_l
\]
be nonempty disjoint closed intervals in $(0,1)$.
Let 
\[
u_1,\dots,u_l \in M_l
\]
be unitaries such that, for each $s$, $u_s \phi_k(d) u_s^*$ is a diagonal matrix whose $s^\text{th}$ diagonal entry is $\lambda_k$.
Let 
\[
h_1,\dots,h_l \in C_0((0,1), [0,1])
\]
be positive normalized functions with disjoint support, such that $h_s|_{I_s} \equiv 1$ for each $s$.
Set $Z := (0,1]^2$ and define
\begin{eqnarray*} 
\lefteqn{\psi_k: C_0(Z,\cO_2)  \cong  C_0((0,1]) \tens C_0((0,1],\cO_2)} \\ 
&& \to   \; C([0,1]) \otimes M_{l} \cong C([0,1],M_l) \subset C([0,1],A) 
\end{eqnarray*}
by
\[ \psi_k(f \tens b) = \sum_{s=1}^l f(h_s) \tens u_s \phi_k(b) u_s^*. \]
Let $f \in C_{\mathrm{c}}((0,1])$ be a function satisfying $f(1)=1$ and set 
\[
c=f \tens d \in C_{\mathrm{c}}(Z, \C\cdot1_{\cO_2}).
\]
By construction,  $\psi_k(c) \in C((0,1), D_l)_+$, and for $t \in I_s$, the $s^\text{th}$ diagonal entry is $\lambda_k$.
Let 
\[
c'_k \in C([0,1],D_l)_+
\]
be of norm $1$, such that 
\[
\|c'_k-\psi_k(c)\| = |1-\lambda_k|
\]
and for $t \in I_s$, the $s^\text{th}$ diagonal entry is $1$.
Feeding 
\[
a_{1/2} := c'_k
\]
to Lemma \ref{TracialSplit}, let 
\[
a_{0,k},a_{1,k},e_{0,k},e_{1/2,k},e_{1,k} \in C([0,1],D_l)_+
\]
be the output, satisfying (i)-(v) of Lemma \ref{TracialSplit}.

Having found these for each $k$, set
\[ \psi := \pi_\infty \circ (\psi_1, \psi_2,\dots): C_0(Z,\cO_2) \to C([0,1],A)_\infty. \]
Set 
\[
a_i := \pi_\infty(a_{i,1},a_{i,2},\dots)
\]
for $i=0,1$ and 
\[
e_i := \pi_\infty(e_{i,1},e_{i,2},\dots)
\]
for $i=0,\frac12,1$.

Since  all unitaries in $M_l$ (and in particular, all $u_s$'s) are in the same path component, $\psi_k$ is unitarily equivalent to $\alpha \tens \phi_k$, where 
\[
\alpha:C_0((0,1]) \to C([0,1])
\]
is the $^*$-homomorphism given by 
\[
f \mapsto f(h_1 + \ldots + h_l).
\]
From this observation and Proposition \ref{AsymptoticTensor}, it follows that $\psi$ is a $^*$-homomorphism.

Notice further that 
\[
\psi(c) = \pi_\infty(c'_1,c'_2,\dots),
\]
and therefore, drawing on the finite stage results, we see that 
\[
a_0+a_1+\psi(c) = 1
\]
and that (i)-(v) hold.
\end{proof}

\begin{lemma}
\label{DimDropUnit}
Let $p,q > 1$ be natural numbers.
Let $X=[0,1]^m$ for some $m$ and let $\e > 0$.

Then there exist positive orthogonal elements
\[ h_0,\dots,h_k \in C(X,\jsZ)_\infty, \]
a $^*$-homomorphism
\[ \phi:C_0(Z, \cO_2) \to C(X,\jsZ)_\infty \]
for some locally compact, metrizable, finite dimensional space $Z$, and a positive element $c \in C_{\mathrm{c}}(Z, \C\cdot1_{\cO_2})$ such that $\phi(c)$ commutes with $h_0,\dots,h_k$,
\[ h_0 + \ldots + h_k + \phi(c) = 1, \]
and the support of $h_i$ has diameter at most $\e$ for $i=0,\dots,k$ with respect to the uniform metric on $[0,1]^{m}$.

In addition, there exist positive contractions $e_0,e_{1/2},e_1 \in C(X,\jsZ)_\infty$ such that
\begin{enumerate}
\item $e_0,e_1$ are orthogonal,
\item $e_0+e_{1/2}+e_1 = 1$,
\item $e_j$ is identically $1$ on $\{j\} \times [0,1]^{m-1}$, for $j=0,1$,
\item $\phi(c)$ acts like a unit on $e_{1/2}$,
\item $e_0+e_1$ acts like a unit on $h_i$ for all $i=0,\dots,k$, 
\item $e_0,e_{1/2},e_1,h_0,\dots,h_k,\phi(c)$ all commute.
\end{enumerate}
\end{lemma}

\begin{proof}
This will be proven in three steps.
In Step 1, we will prove the statement of the proposition with $\jsZ$ replaced by a UHF algebra of infinite type and with $m=1$.
In Step 2, we will still replace $A$ by a UHF algebra of infinite type, but allow any $m \in \mathbb{N}$.
Step 3 will be the proof of the proposition.

\bigskip
\underline{Step 1.}
Let $A$ be a UHF algebra of infinite type.
Let 
\[
a_0,a_1, \psi, c,e'_0,e'_{1/2},e'_1, Z
\]
be as in Lemma \ref{IntervalUnit}, with $e_{i}'$ in place of $e_{i}$.
Note that each $a_i$  has a positive normalized lift 
\[ \textstyle (a_{i,j})_{j=1}^\infty \in \prod_\mathbb{N} C([0,1],A)\]
such that $a_{i,j}(t) = \dl_{i,t} 1$ for all $i,t=0,1$ and all $j$; likewise, each $e_{i}'$, $i=0,\frac12,1$ has a positive normalized lift
\[ \textstyle (e'_{i,j})_{j=1}^\infty \in \prod_\mathbb{N} C([0,1],A)\]
such that, for $i=0,1$, $e_{i,j}'(i) = 1$.

Let $k \geq 2/\e$ be a natural number.
For $i=0,\dots,k$, $j \in \mathbb{N}$, and $t \in [0,1]$, set
\begin{equation}
\label{DimDropUnit-1hDefn}
h_{i,j}(t) := \begin{cases}
0, \quad &\text{if } t \leq \frac{i-1}{k} \text{ or } t \geq \frac{i+1}k, \\
a_{1,j}(k t - (i-1)), \quad &\text{if } t \in \left[\frac{i-1}k, \frac ik\right], \\
a_{0,j}(k t - i), \quad &\text{if } t \in \left[\frac ik, \frac{i+1}k\right].
\end{cases}
\end{equation}
Note that the endpoint conditions on $a_{i,j}$ make $h_{i,j}$ well-defined and continuous on $[0,1]$.
Likewise, set
\begin{equation}
\label{DimDropUnit-1eDefn}
e_{i,j}(t) := \begin{cases}
e_{i,j}'(0), \quad &\text{if } t = 0, \\
e_{i,j}'(1), \quad &\text{if } t \geq \frac1k, \\
e_{i,j}'(kt), \quad &\text{if } t \in \left[0, \frac1k\right].
\end{cases}
\end{equation}
Set 
\[
h_i := \pi_\infty(h_{i,1},h_{i,2},\dots), e_i := \pi_\infty(e_{i,1},e_{i,2},\dots)  \in C([0,1],A)_\infty
\]
for $i=0,1$. Choose a c.p.c.\ lift for $\psi$, i.e., c.p.c.\ maps 
\[
\psi_j:C_0(Z,\cO_2) \to C_0((0,1),A) \subset C([0,1],A)
\]
such that 
\[
\psi := \pi_\infty \circ (\psi_1,\psi_2,\dots).
\]
Define 
\[
\phi_j:C_0(Z,\cO_2) \to C([0,1],A)
\]
by
\begin{equation}
\label{DimDropUnit-1phijDefn}
\phi_j(a)(t) = \psi_j(a)(k t - i),
\end{equation}
if $i \in \mathbb{N}$ is such that $t \in \left[\frac ik, \frac{i+1}k\right]$.
Note that this is well-defined since the image of $\psi_j$ is contained in $C_0((0,1),A)$.
Use $(\phi_j)_{j=1}^\infty$ to define
\begin{equation*}
\phi = \pi_\infty \circ (\phi_1,\phi_2,\dots): C_0(Z,\cO_2) \to C([0,1],A)_\infty.
\end{equation*}
Then $\phi$ is a $^*$-homomorphism.

Let us first show that $h_0 + \ldots + h_k + \phi(c) = 1$, and then that (i)-(vi) hold.
For $t \in [0,1]$, let $i$ be such that $t \in \left[\frac ik, \frac{i+1}k\right]$.
Then by \eqref{DimDropUnit-1hDefn}, we have for all $i$, 
\[
h_i(t) = a_0(kt - i),\; h_{i+1}(t) = a_1(kt-i) \mbox{ and } h_j(t) = 0
\]
for $j\neq i,i+1$.
Thus,
\begin{eqnarray*}
(h_0 + \ldots + h_k + \phi(c))(t) &\stackrel{\eqref{DimDropUnit-1phijDefn}}{=}& a_0(kt - i) + a_1(kt - i) + \psi(c)(kt-i) \\
&\stackrel{\eqref{IntervalUnit-POU}}{=}& 1.
\end{eqnarray*}

Properties (i) and (ii) hold by Lemma \ref{IntervalUnit} (i) and (ii), and since for each $t \in [0,1]$, there exists $s$ such that $e_j(t) = e'_j(s)$ for $j=0,\frac12,1$ (by \eqref{DimDropUnit-1eDefn}).
Property (iii) holds since $e_i(i) = e'_i(i)$ (by \eqref{DimDropUnit-1eDefn}) and since $a_i(i)=1$.

(iv): $e_{1/2}$ is supported on $\left[0,\frac1k\right]$, so it suffices to show that 
\[
(\phi(c)e_{1/2})(t) = e_{1/2}(t)
\]
for $t \in \left[0,\frac1k\right]$. 
But for such $t$,
\begin{eqnarray*}
(\phi(c)e_{1/2})(t) &\stackrel{\eqref{DimDropUnit-1eDefn},\eqref{DimDropUnit-1phijDefn}}{=}& \psi(c)(kt)e'_{1/2}(kt) \\
&\stackrel{\text{Lemma \ref{IntervalUnit} (iii)}}{=}& e'_{1/2}(kt)\\
&\stackrel{\eqref{DimDropUnit-1eDefn}}{=}& e_{1/2}(t).
\end{eqnarray*}

(v): By a similar computation  (this time using Lemma \ref{IntervalUnit} (iv)), we see that $e_0a_0 = a_0$, while $e_1a_i=a_i$ for $i=1,\dots,k$.

(vi) is clear from \eqref{DimDropUnit-1hDefn}, \eqref{DimDropUnit-1eDefn}, \eqref{DimDropUnit-1phijDefn}, and Lemma \ref{IntervalUnit} (v).

Finally, also, for each $i$, the support of $h_{i}$ is contained in $\left[\frac{i-1}k,\frac{i+1}k\right]$, which has diameter at most $\e$.

\bigskip
\underline{Step 2.}
From Step 1, let
\[ g_0,\dots,g_{k'} \in C([0,1],A)_\infty \]
be orthogonal positive contractions,
\[ \psi:C_0(Y, \cO_2) \to C([0,1],A)_\infty \]
be a $^*$-homomorphism for some locally compact, metrizable, finite dimensional space $Y$, and $d \in C_{\mathrm{c}}(Y, \C\cdot1_{\cO_2})$ a positive contraction such that $\psi(d)$ commutes with $g_0,\dots,g_{k'}$,
\[ g_0 + \ldots + g_{k'} + \psi(d) = 1 \]
and the support of $g_i$ has diameter at most $\e$ for $i=0,\dots,k'$; 
furthermore, let 
\[
e'_0,e'_{1/2},e'_1 \in C([0,1],A)_\infty
\]
be such that 
\begin{enumerate}
\item[(i')] $e'_0,e'_1$ are orthogonal,
\item[(ii')] $e'_0+e'_{1/2}+e'_1 = 1$,
\item[(iii')] $e'_j$ is identically $1$ on $\{j\} \times [0,1]^{m-1}$, for $j=0,1$,
\item[(iv')] $\psi(d)$ acts like a unit on $e'_{1/2}$,
\item[(v')] $e'_0+e'_1$ acts like a unit on $g_i$ for all $i=0,\dots,k'$, 
\item[(vi')] $e'_0,e'_{1/2},e'_1,g_0,\dots,g_{k'},\psi(d)$ all commute.
\end{enumerate}

For $i=(i_1,\dots,i_m) \in \{0,\dots,k'\}^m$, set
\[ h_i := g_{i_1} \tens \ldots \tens g_{i_m} \in (C([0,1],A)^{\tens m})_\infty, \]
where we have used the canonical inclusion 
\[
(C([0,1],A)_{\infty})^{\tens m} \to  (C([0,1],A)^{\tens m})_\infty,
\]
cf.\ Proposition \ref{AsymptoticTensor}.

Then $\{h_i\}$ is a set of pairwise orthogonal positive contractions, and each one has support with diameter at most $\e$ (recall that we are using the uniform metric on $[0,1]^m$).
Proposition \ref{AsymptoticTensor} gives us a $^*$-homomorphism 
\[ \phi' := (\psi ^\sim)^{\tens m}: C := (C_0(Y,\cO_2)^\sim)^{\tens m} \to \left(C([0,1],M_{n^\infty})^{\tens m}\right)_\infty. \]
Set
\[ c := 1-(1-d)^{\tens m} \in C. \]
We can easily see that $\phi'(c)$ commutes with each $h_i$; a simple computation shows that 
\[ \sum_i h_i + \phi'(c) = 1. \]
Setting
\[ \textstyle e_i := e'_i \tens 1^{\tens (m-1)} \quad \text{for }i=0,\frac12,1, \]
it is easy to see that (i),(ii),(iii),(v), and (vi) hold (with $\phi'$ in place of $\phi$).
To see that (iv) holds, we compute
\begin{eqnarray*}
\displaystyle \phi'(c)e_{1/2} &=& (1-(1-\psi(d))^{\tens m})(e'_{1/2} \tens 1^{\tens (m-1)}) \\
\displaystyle &=& \phi'(c) - (e'_{1/2} - \psi(d)e'_{1/2}) \tens (1-\psi(d))^{\tens (m-1)} \\
\displaystyle &\stackrel{\text{(iv')}}{=}& \phi'(c)
\end{eqnarray*}

 We may set 
\[
k := (k'+1)^{m} - 1 
\]
and relabel the $h_{i}$ as $h_{0},\ldots h_{k}$. 

All that remains is to modify $\phi'$ to make it a map whose domain is $C_0(Z,\cO_2)$ for some $Z$.
Set 
\[
Z' := (Y \dunion \{\infty\})^{\times m}.
\]
Then $C$ may be identified with a certain $C(Z')$-subalgebra of $C(Z', \cO_2^{\tens m})$.
All of the fibres of $C$ are isomorphic to $\cO_2$ except for the fibre at $(\infty,\dots,\infty)$, which is $\mathbb{C}$.
One can easily verify that the element $c$ is in $C_0(U,\C\cdot1_{\cO_2^{\tens m}})$ where $U$ is some open subset of $Z'$ whose closure does not contain $(\infty,\dots,\infty)$.
Let $Z$ be an open subset of $Z'$ such that $\overline{U} \subset Z$ and whose closure does not contain $(\infty,\dots,\infty)$; in particular, $\overline{Z}$ is a compact subset of $Z' \setminus \{(\infty,\dots,\infty)\}$.
By Proposition \ref{O2CtsField}, $C_0(Z)C \iso C_0(Z,\cO_2)$ as $C_0(Z)$-algebras.
With this identification, we have $c \in C_{\mathrm{c}}(Z,\C\cdot1_{\cO_2})$ (since $c$ is in the image of the structure map, which is fixed by the isomorphism $C_0(Z)C \iso C_0(Z,\cO_2)$), and we may define
\[ \phi := \phi'|_{C_0(Z)C}:C_0(Z,\cO_2) \to C(X,A)_\infty. \]

\bigskip
\underline{Step 3.}
Let $p_0,p_1$ be coprime natural numbers.
Since $\jsZ_{p_0^\infty,p_1^\infty}$ (as defined in \ccite{Section 2}{RordamWinter:Z}) embeds unitally into $\jsZ$ (\ccite{Proposition 2.2}{Rordam:Z}), it suffices to do this part with $\jsZ_{p_0^\infty,p_1^\infty}$ in place of $\jsZ$.

From Step 2, for $i=0,1$, we may find 
\[ h^{(i)}_0,\dots,h^{(i)}_k \in C(X,M_{p_i^\infty})_\infty, \]
a $^*$-homomorphism
\[ \phi_i:C_0(Z_i, \cO_2) \to C(X,M_{p_i^\infty})_\infty \]
for some locally compact, metrizable, finite dimensional space $Z_i$, and a positive element 
\[
c_i \in C_{\mathrm{c}}(Z_i, \C\cdot1_{\cO_2})
\]
such that $\phi_{i}©$ commutes with $h^{(i)}_0,\dots,h^{(i)}_k$,
\[ h^{(i)}_0 + \ldots + h^{(i)}_k + \phi_i(c_i) = 1, \]
and the support of $h^{(i)}_j$ has diameter at most $\e$ for $j=1,\dots,k$.
We may also find $e^{(i)}_l$ for $l=0,\frac12, 1$ satisfying (i)-(vi).

From Lemma \ref{IntervalUnit}, let 
\[ \textstyle a_0,a_1,e'_0,e'_{1/2},e'_1 \in C\left(\left[\frac13, \frac23\right],A\right)_\infty \]
be positive orthogonal contractions,
\[ \textstyle \psi:C_0(Y, \cO_2) \to C_0\left(\left(\frac13,\frac23\right),M_{(p_0p_1)^\infty}\right)_\infty \]
be a $^*$-homomorphism, for some locally compact, metrizable, finite dimensional space $Y$, and let 
\[
d \in C_{\mathrm{c}}(Z, \C\cdot1_{\cO_2})
\]
be positive such that $\psi(d)$ commutes with $a_0,a_1$,
\[ a_0 + a_1 + \psi(d) = 1, \]
$a_0\left(\frac13\right)=a_1\left(\frac23\right)=1$, and such that (i)-(v) of Lemma \ref{IntervalUnit} hold.
We continuously extend $a_0,a_1,e'_0,e'_{1/2},e'_1$ to $[0,1]$ by allowing them to be constant on $\left[0,\frac13\right]$ and on $\left[\frac23,1\right]$.

Upon choosing an isomorphism
\[
M_{(p_{0}p_{1})^{\infty}} \otimes M_{p_{0}^{\infty}} \otimes M_{p_{1}^{\infty}} \cong M_{p_{0}^{\infty}} \otimes M_{p_{1}^{\infty}}
\] 
we obtain a $^{*}$-homomorphism
\begin{eqnarray*}
\lefteqn{\rho:  C([0,1],M_{(p_0p_1)^\infty})_\infty \tens C(X,M_{p_0^\infty})_\infty \tens C(X,M_{p_1^\infty})_\infty }\\
&&\to \;C([0,1] \times X, M_{(p_0 p_1)^\infty} \tens M_{p_0^\infty} \tens M_{p_1^\infty})_\infty \\
&&\iso \;C([0,1] \times X, M_{p_0^\infty} \tens M_{p_1^\infty})_\infty,
\end{eqnarray*}
and define
\begin{align*}
\hat{h}_{0,j} &:= \rho(a_0 \tens h^{(0)}_j \tens 1_{C(X,M_{p_1^\infty})_\infty}) \quad \text{and} \\
\hat{h}_{1,j} &:= \rho(a_1 \tens 1_{C(X,M_{p_0^\infty})_\infty} \tens h^{(1)}_j)
\end{align*}
for $j=1,\dots,k$.
Note that $a_i$ has a lift 
\[
\textstyle
(a_{i,k})_{k=1}^\infty \in \prod_\mathbb{N} C([0,1],M_{(p_0p_1)^\infty})
\]
such that $a_{i,k}(t) \in \C \cdot 1$ for $t=0,1$, and consequently,
\[ \hat{h}_{i,j} \in C(X,\jsZ_{p_0^\infty,p_1^\infty})_\infty. \]
Define a $^{*}$-homomorphism
\begin{eqnarray*}
\lefteqn{\phi :=  \rho \circ (\id_{C([0,1])} \cdot (\psi^\sim) \tens (\phi_0^\sim) \tens (\phi_1^\sim)): }\\
&&C([0,1]) \tens C_0(Y,\cO_2)^\sim \tens C_0(Z_0,\cO_2)^\sim \tens C_0(Z_1,\cO_2)^\sim \\
& & \to  C([0,1] \times X, M_{p_0^\infty} \tens M_{p_1^\infty})_\infty .
\end{eqnarray*}
Let 
\[
Y' := \{y \in Y: d(y) > 0\}
\]
and 
\[
Z_i' := \{z \in Z_i: c_i(z) \neq 0\},
\]
and using these, set
\begin{align*}
C :=  \mathrm{C}^{*}\Big(&C_0\left[0,1\right) \tens 1_{C_0(Y,\cO_2)^\sim} \tens C_0(Z_0', \cO_2) \tens 1_{C_0(Z_1,\cO_2)^\sim}, \\
&C_0\left(0,1\right] \tens 1_{C_0(Y,\cO_2)^\sim} \tens 1_{C_0(Z_0,\cO_2)^\sim} \tens C_0(Z_1',\cO_2) , \\
&1_{C([0,1])} \tens C_0(Y',\cO_2) \tens 1_{C_0(Z_0,\cO_2)^\sim} \tens 1_{C_0(Z_1,\cO_2)^\sim}\Big). 
\end{align*}
Using Proposition \ref{O2CtsField} as in Step 2, $C$ is a subalgebra of some $C_0(Z)$-algebra
\[ D \subset C[0,1] \tens C_0(Y,\cO_2)^\sim \tens C_0(Z_0, \cO_2) \tens C_0(Z_1, \cO_2), \]
for some open subset $Z$ of
\[ [0,1] \times (Y' \cup \{\infty\}) \times (Z'_0 \cup \infty) \times (Z'_1 \cup \infty), \]
and $D$ is isomorphic, as a $C_0(Z)$-algebra, to $C_0(Z,\cO_2)$, via an isomorphism taking $C$ into $C_c(Z,\cO_2)$.
One easily sees that $\phi(C) \subset C(X,\jsZ_{p_0,p_1})$.

Let $f_0 \in C_0[0,1)_+$ be identically $1$ on $\left[0,\frac23\right]$, and let $f_1 \in C_0(0,1]_+$ be identically $1$ on $\left[\frac13,1\right]$.
Set
\[ \hat c := f_0 \tens 1 \tens c_0 \tens 1 + f_1 \tens 1 \tens 1 \tens c_1 + 1 \tens d \tens 1 \tens 1 \in C. \]
Identifying $D$ with $C_0(Z,\cO_2)$, we see that $\hat c \in C_c(Z,\C\cdot1_{\cO_2})$.
It is straightforward to check  that $\phi(\hat c)$ commutes with $\hat{h}_{i,j}$ for all $i,j$, and we may easily compute
\[ \phi(\hat c) + \sum_{i,j} \hat{h}_{i,j} \geq 1. \]
Let $g \in C_0(0,\infty]$ be the function $g(t)=\max\{t,1\}$ and set 
\[ c:= g(\hat c). \]
Then by commutativity, it follows that
\begin{equation}
\label{DimDropUnit-3SuperUnit}
 \phi(c) + \sum_{i,j} \hat{h}_{i,j} \geq 1.
\end{equation}

Let $g_0,g_{1/2},g_1 \in C(X)_+$ be a partition of unity such that $g_j$ is identically $1$ on $\{j\} \times [0,1]^{m-1}$ for $j=0,1$, $g_0$ is
supported on $\left[0,\frac13\right] \times [0,1]^{m-1}$, and $g_1$ is supported on $\left[\frac23,1\right] \times
[0,1]^{m-1}$. 
Let us define
\begin{equation}
\label{DimDropUnit-3eDefn}
 e_j := \rho(e_0' \tens e^{(0)}_j \tens 1 + e_1' \tens 1 \tens e^{(1)}_j) + g_j\rho(e_{1/2}' \tens 1 \tens 1)
\end{equation}
for $j=0,\frac12,1$.
It is clear by their definitions that $e_0,e_{1/2},e_1,\hat h_0,\dots,\hat h_k,\phi(c)$ all commute.

Let us now check that $(e_0+e_1)\hat{h}_{i,j} = \hat{h}_{i,j}$.
Certainly,
\begin{eqnarray*}
\displaystyle (e_0+e_1)\hat{h}_{0,j} &\stackrel{\eqref{DimDropUnit-3eDefn}}{=}& \Big( \rho\Big(e_0' \tens (e_0^{(0)} + e_1^{(0)}) \tens 1
+ e_1' \tens 1 \tens (e_0^{(1)} \tens e_1^{(1)})\Big) \\
\displaystyle &&+ (g_0+g_1)\rho(e_{1/2}' \tens 1 \tens 1) \Big) \rho(a_0 \tens h_j^{(0)} \tens 1) \\
\displaystyle &\stackrel{\begin{tabular}{@{}c} {\scriptsize \text{Lemma \ref{IntervalUnit}}} \\ {\scriptsize \text{(ii),(iv)}} \end{tabular}}{=}& \rho\left(a_0 \tens \left((e_0^{(0)} + e_1^{(0)})h_j^{(0)}\right) \tens 1\right) \\
\displaystyle &\stackrel{\text{Step 2 (v)}}{=}& \rho(a_0 \tens h_j^{(0)} \tens 1) = \hat{h}_{0,j},
\end{eqnarray*}
and likewise, $(e_0+e_1)\hat{h}_{1,j} = \hat{h}_{1,j}$ as required.

Since all terms in \eqref{DimDropUnit-3SuperUnit} commute, it is easy to see that for any $\e > 0$, there exist orthogonal elements $\tilde{h}_{i,j} \leq \hat{h}_{i,j}$ which commute with $e_0,e_{1/2},e_1$ and $\phi(c)$, such that
\[ \phi(c) + \sum_{i,j} \tilde{h}_{i,j} =_{\e} 1. \]
Then, by a diagonal sequence argument, it follows that there exist orthogonal elements $h_{i,j}$ with supports contained in those of $\hat{h}_{i,j}$ which commute with $e_0,e_{1/2},e_1$ and $\phi(c)$, such that
\[ \phi(c) + \sum_{i,j} h_{i,j} = 1, \]
and
\[ (e_0+e_1)h_{i,j} = h_{i,j}. \]
Hence, (v) holds.

Now let us verify (i)-(iv).

(i) holds using the following orthogonalities:
\begin{align*}
 e_0^{(i)} &\perp e_1^{(i)}, \quad \text{i=0,1} \\
g_0 &\perp g_1 \\
e_0' &\perp e_1' \\
\rho(1 \tens e_j^{(0)} \tens 1) &\perp g_{1-j}, \quad \text{j=0,1}, \\
\rho(1 \tens 1 \tens e_j^{(1)}) &\perp g_{1-j}, \quad \text{j=0,1}.
\end{align*}

(ii): We compute
\begin{eqnarray*}
\displaystyle \e_0 + e_{1/2} + e_1 &\stackrel{\eqref{DimDropUnit-3eDefn}}{=}& \rho\Big(e_0' \tens \left(e_0^{(0)} + e_{1/2}^{(0)} + e_1^{(0)}\right) \tens 1 \\
&& \quad + e_1' \tens 1 \tens \left(e_0^{(1)} + e_{1/2}^{(1)} + e_1^{(1)} \right)\Big) \\ 
\displaystyle && \quad +(g_0 + g_{1/2} + g_1)\rho(e_{1/2}' \tens 1 \tens 1) \\
\displaystyle &\stackrel{\text{Step 2 (ii)}}{=}& \rho\left(\left(e_0' + e'_1 + e'_{1/2}\right) \tens 1 \tens 1\right) \\
\displaystyle &\stackrel{\text{Lemma \ref{IntervalUnit}} \text{(ii)} }{=}& 1
\end{eqnarray*}

(iii): For $x \in \{j\} \times [0,1]^{m-1}$,
\begin{eqnarray*}
\displaystyle e_j(x) &\stackrel{\text{Step 2 (iii)}}{=}& e_0' + e_1' + g_j(x)e_{1/2}' \\
\displaystyle &\stackrel{\text{Lemma \ref{IntervalUnit} (ii)}}{=}& 1.
\end{eqnarray*}

(iv) follows from the fact that $\phi(\hat{c})e_{1/2} = e_{1/2}\phi(\hat{c}) \geq e_{1/2}$, by considering irreducible representations of $\mathrm{C}^*(\phi(\hat{c}),e_{1/2})$.
\end{proof}

\begin{proof}[Proof of Theorem \ref{AHZdr}]
By Proposition \ref{drPermanence} (i) and \ccite{Proposition 3.8}{KirchbergWinter:CovDim}, it suffices to verify the theorem for $\mathrm{C}^*$-algebras $A$ of the form $C(X,\K)$ where $X$ is compact and Hausdorff.
By \ccite{(3.5)}{KirchbergWinter:CovDim}, it suffices to prove it for $A = C(X)$.
Again by Proposition \ref{drPermanence} (i), it suffices to assume that $C(X)$ is finitely generated.
Finally, when $C(X)$ is finitely generated, it is a quotient of $C([0,1]^m)$ for some $m$, and so by \ccite{(3.3)}{KirchbergWinter:CovDim}, the result reduces to showing that $\dr C(X,\jsZ) \leq 2$ for $X=[0,1]^m$.
By Proposition \ref{SSA-TensorFormula}, we must show that the first factor embedding $C(X,\jsZ) \to C(X,\jsZ) \tens \jsZ$ has decomposition rank at most 2.

We will do this in two steps.
In Step 1, we will use Lemma \ref{DimDropUnit} to show that the first factor embedding $\iota_0:C(X) \to C(X) \tens \jsZ$ has decomposition rank at most 2.
In Step 2, we will use Step 1, with $X$ replaced by $X \times [0,1]$, to prove the theorem.

\bigskip
\underline{Step 1.}
Due to Proposition \ref{AsympSequence}, it suffices to replace $\iota_0$ by its composition with the inclusion $C(X) \tens \jsZ \subset (C(X) \tens \jsZ)_\infty$; that is, $\iota_{0}$ is now
\[ C(X) \iso C(X) \tens 1_\jsZ \subset C(X) \tens \jsZ \subset (C(X) \tens \jsZ)_\infty. \]
To show that $\dr \iota_0 \leq 2$, we verify condition (iv) of Proposition \ref{AbelianInclusionDim}.
Let $\mathcal{U}$ be an open cover of $X$ and let $\e > 0$.
By the Lebesgue Covering Lemma, we may possibly reduce $\e$ so that $\mathcal{U}$ is refined by the set of all open sets of diameter at most $\e$.
Then, it suffices to assume that $\mathcal{U}$ is in fact the set of all open sets of diameter at most $\e$.

Let $h_0,\dots,h_k,\phi,c$ be as in Lemma \ref{DimDropUnit}.
By Theorem \ref{O2dimnuc} and condition (iv) of Proposition \ref{AbelianInclusionDim}, we may find 
\[
b_j^{(i)} \in C_0(X \times Z, \cO_2) \cong C(X) \otimes C_{0}(Z) \otimes \mathcal{O}_{2}
\]
 for $i=0,1, j=0,\dots,r$ such that
\begin{enumerate}
\item for each $i=0,1$, the elements $b_0^{(i)},\dots,b_r^{(i)}$ are pairwise orthogonal,
\item for each $i,j$, the support of $b_j^{(i)}$ is contained in $U \times Z$ for some $U \in \mathcal{U}$,
\item $\big\| \sum_{i,j} b_j^{(i)} - 1_{C(X)} \tens c \big\| < \e$ (note that $c \in C_{0}(Z) \otimes 1_{\mathcal{O}_{2}}$).
\end{enumerate}

Define
\[ \hat \phi:C_{0}(X \times Z, \cO_2) \iso C(X) \tens C_0(Z,\cO_2) \to C(X,\jsZ)_\infty \]
by $\hat \phi(f \tens a) = f\phi(a)$.
This is a $^*$-homomorphism.
For $i=0,1$ and $j=0,\dots,r$, set
\[ a^{(i)}_j := \hat \phi(b_j^{(i)}), \]
and, for $j=0,1\dots,k$, set
\[ a^{(2)}_j := h_j. \]

Since $\hat \phi$ is a homomorphism, $a^{(i)}_0,\dots,a^{(i)}_r$ are pairwise orthogonal for $i=0,1$.
Also, by the definition of $\hat \phi$ and the choice of $b^{(i)}_j$, the support of each $a^{(i)}_j$ is contained in some set in $\mathcal{U}$, for $i=0,1$. Since the supports of the $h_{j}$ have diameter at most $\epsilon$, the respective statement holds for the $a^{(2)}_{j}$ as well.
Finally,
\[
\begin{array}{rll}
\displaystyle  \sum_{i,j} a^{(i)}_j &=&\displaystyle \hat \phi \Big(\sum_{i=0,1} \sum_{j=0}^k b_j^{(i)} \Big) + \sum_{j=0}^k h_j \\
&=_{\e}&\displaystyle \hat \phi(1 \tens c) + \sum_{j=0}^k h_j \\
&=& \displaystyle \phi(c) + \sum_{j=0}^k h_j \\
&=& 1,
\end{array}
\]
as required.

\bigskip
\underline{Step 2.}
Since $\jsZ$ is an inductive limit of algebras of the form $\jsZ_{p,q}$ (for $p,q \in \mathbb{N}$), by Proposition \ref{drPermanence} (i), it suffices to show that the decomposition rank of the first factor embedding
\begin{equation}
\label{DimNucCXZEmbedding2-NewIota}
\iota:= \id_{C(X, \jsZ_{p,q})} \otimes 1_{\mathcal{Z}}: C(X, \jsZ_{p,q}) \to C(X, \jsZ_{p,q}) \tens \jsZ
\end{equation}
is at most $2$. The proof will combine Step 1 with the idea of Proposition \ref{AbelianInclusionDim} (iv) $\Rightarrow$ (iii).

For $t \in [0,1]$, we let $\ev_t:\jsZ_{p,q} \to M_p \tens M_q$ denote the point-evaluation at $t$, while we also let
\begin{align*}
\overline{\ev}_0:&\jsZ_{p,q} \to M_p, \\
\overline{\ev}_1:&\jsZ_{p,q} \to M_q
\end{align*}
denote the irreducible representations which satisfy
\begin{align*}
\ev_0(\cdot) &= \overline{\ev}_0(\cdot) \tens 1_{M_q} \text{ and} \\
\ev_1(\cdot) &= 1_{M_p} \tens \overline{\ev}_1(\cdot).
\end{align*}

Let $\F \subset C(X,\jsZ_{p,q})$ be the finite set to approximate, and let $\e > 0$ be the tolerance.
Let us assume that $\F$ consists of contractions.
Let $\mathcal{U}$ be an open cover of $X \times [0,1]$, such that, for all $f \in \F$ and $U \in \mathcal{U}$, if $(x,t),(x',t') \in U$ then
\[ \|\ev_t(f(x)) - \ev_{t'}(f(x'))\| < \e/2. \]
Let us also assume that no $U \in \mathcal{U}$ intersects both $X \times \{0\}$ and $X \times \{1\}$.

Using Step 1 (with $X \times [0,1]$ in place of $X$) and Proposition \ref{AbelianInclusionDim} (iv), we may find a $3$-colourable $\frac{\epsilon}{2}$-approximate partition of unity
\[ (a^{(i)}_j)_{i=0,1,2; j=0,\dots,r} \subset C(X \times [0,1]) \tens \jsZ \]
subordinate to $\mathcal{U}$, and such that
\[ \sum a^{(i)}_j \leq 1. \]
Upon replacing $\mathcal{U}$ by a subcover if necessary, we may clearly assume that $\mathcal{U}$ is of the form $(U^{(i)}_{j})_{i=0,1,2; j=0,\dots,r}$, with the support of each $a^{(i)}_{j}$ being contained in $U^{(i)}_{j}$.

For each $i,j$, we shall choose a matrix algebra $F^{(i)}_j$ and produce maps
\[
C(X, \jsZ_{p,q}) \labelledrightarrow{\psi^{(i)}_j} F^{(i)}_j \labelledrightarrow{\phi^{(i)}_j} C(X,\jsZ_{p,q}) \tens \jsZ.
\]
We distinguish three cases, depending on properties of the set $U^{(i)}_j \in \mathcal{U}$.
In every case, we arrange that
\[ \phi^{(i)}_j\psi^{(i)}_j(f) = a^{(i)}_j \tens \ev_{t^{(i)}_j}(f(x^{(i)}_j)), \]
where $(x^{(i)}_j, t^{(i)}_j)$ is a point from $U^{(i)}_j$, and we make sense of the right-hand side by using the canonical identification of $C(X,\jsZ_{p,q}) \tens \jsZ$ with a subalgebra of 
\[ C(X \times [0,1]) \tens \jsZ \tens M_p \tens M_q \]
(determined by boundary conditions at $X \times\{0\}$ and at $X \times \{1\}$).

\bigskip
\underline{Case 1.} If $U^{(i)}_j \cap \left(X \times \{0\}\right) \neq \emptyset$, then let $(x^{(i)}_j, t^{(i)}_j = 0)$ be a point in this intersection.
We set $F^{(i)}_j := M_p$ and define
\begin{align*}
\psi^{(i)}_j(f) &= \overline{\ev}_0(f(x^{(i)}_j)), \\
\phi^{(i)}_j(T) &= a^{(i)}_j \tens T \tens 1_{M_q},
\end{align*}
By assumption $U^{(i)}_j \cap \left(X \times \{1\}\right) = \emptyset$, so for all $x \in X$,
\[ \ev_1(\phi^{(i)}_j(T)(x)) = 0, \]
and therefore, the range of $\phi^{(i)}_j$ lies in $C(X,\jsZ_{p,q}) \tens \jsZ$.

\bigskip
\underline{Case 2.} If $U^{(i)}_j \cap \left(X \times \{1\}\right) \neq \emptyset$, then as in Case 1, let $(x^{(i)}_j, t^{(i)}_j = 1)$ be a point in this intersection.
We set $F^{(i)}_j := M_q$ and define
\begin{align*}
\psi^{(i)}_j(f) &= \overline{\ev}_1(f(x^{(i)}_j)), \text{ and} \\
\phi^{(i)}_j(T) &= a^{(i)}_j \tens 1_{M_p} \tens T.
\end{align*}

\bigskip
\underline{Case 3.} If $U^{(i)}_j \cap \left(X \times \{0\}\right) = \emptyset$ and $U^{(i)}_j \cap \left(X \times \{1\}\right) = \emptyset$, then let $(x^{(i)}_j, t^{(i)}_j)$ be any point in $U^{(i)}_j$.
We set $F^{(i)}_j := M_p \tens M_q$ and define
\begin{align*}
\psi^{(i)}_j(f) &= \ev_{t^{(i)}_j}(f(x^{(i)}_j)), \text{ and} \\
\phi^{(i)}_j(T) &= a^{(i)}_j \tens T.
\end{align*}

We now set $F = \bigdsum_{i,j} F^{(i)}_j$ and use $(\psi^{(i)}_j)$ and $(\phi^{(i)}_j)$ to define 
\[ C(X,\jsZ_{p,q}) \labelledrightarrow{\psi} F \labelledrightarrow{\phi} C(X,\jsZ_{p,q}) \tens \jsZ. \]
We have that $\psi$ is c.p.c.\ since all of its components are.
Each $\phi^{(i)}_j$ is c.p.\ and order zero.
For each $i,j_1,j_2$, if $j_1 \neq j_2$ then the images of $\phi^{(i)}_{j_1}$ and $\phi^{(i)}_{j_2}$ are orthogonal.
Thus, for each $i$,
\[ \phi|_{\bigdsum_j F^{(i)}_j} \]
is order zero.
Also, $\phi(1) = \sum a^{(i)}_j \leq 1$, so that $\phi$ is contractive.

Finally, let $f\in \F$ and let us check that $\phi\psi(f) =_\e f$.
As in the proof of Proposition \ref{AbelianInclusionDim} (iv) $\Rightarrow$ (iii), we have for each $i,j$ that if $x \in U^{(i)}_j$ then
\[ \ev_{t^{(i)}_j}(f(x^{(i)}_j)) =_{\e/2} \ev_{t}(f(x)), \]
and therefore,
\[ \ev_{t}(f(x)) - \frac{\e}{2} \cdot 1_{M_p \tens M_q} \leq \ev_{t^{(i)}_j}(f(x^{(i)}_j)) \leq \ev_{t}(f(x)) + \frac{\e}{2} \cdot 1_{M_p \tens M_q}. \]

Since $a^{(i)}_j$ commutes with $f$, this gives
\begin{align*}
 a^{(i)}_j(x,t)\left(\ev_{t}(f(x)) - \frac{\e}{2} \cdot 1_{M_p \tens M_q}\right) &\leq a^{(i)}_j(x,t)\ev_{t^{(i)}_j}(f(x^{(i)}_j)) \\
&\leq a^{(i)}_j(x,t)\left(\ev_{t}(f(x)) + \frac{\e}{2} \cdot 1_{M_p \tens M_q}\right).
\end{align*}
Moreover, since $a^{(i)}_j$ vanishes outside of $U^{(i)}_j$, these inequalities continue to hold for all $x \in X$ and all $t \in [0,1]$.

Summing over $i,j$, we find that
\begin{eqnarray*}
\lefteqn{\sum_{i,j} a^{(i)}_j(x,t)\left(\ev_{t}(f(x)) - \frac{\e}{2} \cdot 1_{M_p \tens M_q}\right)} \\
&\leq& \sum_{i,j} a^{(i)}_j(x,t) \ev_{t^{(i)}_j}(f(x^{(i)}_j)) \\
&\leq& \sum_{i,j} a^{(i)}_j(x,t)\left(\ev_{t}(f(x)) + \frac{\e}{2} \cdot 1_{M_p \tens M_q}\right)
\end{eqnarray*}
and therefore,
\[
\begin{array}{rll}
\ev_t(f(x)) &=_{\e/2}&\displaystyle \sum_{i,j} a^{(i)}_j(x,t) \ev_t(f(x)) \\
&=_{\e/2}&\displaystyle \sum_{i,j} a^{(i)}_j(x,t) \ev_{t^{(i)}_j}(f(x^{(i)}_j)) \\
&=&\displaystyle \ev_t(\phi\psi(f)(x)).
\end{array}
\]
Since this holds for all $x \in X,t \in [0,1]$, this means that $\|f - \phi\psi(f)\| < \e$, as required.
\end{proof}

\end{document}